\documentclass[11pt,a4paper]{article}

\usepackage{epsf,epsfig,amsfonts,amsgen,amsmath,amstext,amsbsy,amsopn,amsthm,cases,listings,color
}
\usepackage{ebezier,eepic}
\usepackage{color}
\usepackage{multirow}
\usepackage{epstopdf}
\usepackage{graphicx}
\usepackage{pgf,tikz}
\usepackage{mathrsfs}
\usepackage[marginal]{footmisc}
\usepackage{enumitem}
\usepackage{appendix}

\usepackage{pgf,tikz,pgfplots}
\usepackage{subfigure}
\pgfplotsset{compat=1.13}
\usepackage{mathrsfs}
\usetikzlibrary{arrows}

\definecolor{uuuuuu}{rgb}{0.27,0.27,0.27}
\definecolor{sqsqsq}{rgb}{0.1255,0.1255,0.1255}

\newtheorem{dfn}{Definition} [section]
\newtheorem{theorem}[dfn]{Theorem}
\newtheorem{lemma}[dfn]{Lemma}

\newtheorem{conjecture}[dfn]{Conjecture}
\newtheorem{claim}[dfn]{Claim}

\newenvironment{pf}{\noindent{\bf Proof}}

\def\lf{\left\lfloor}

\def\rf{\right\rfloor}

\begin{document}

\title{\bf\Large $d$-cluster-free sets with a given matching number}

\date{\today}

\author{Xizhi Liu \thanks{Department of Mathematics, Statistics, and Computer Science, University of Illinois, Chicago, IL, 60607 USA.\ Email: xliu246@uic.edu}}

\maketitle

\begin{abstract}
Let $3\le d\le k$ and $\nu\ge 0$ be fixed and $\mathcal{F}\subset\binom{[n]}{k}$.
The matching number of $\mathcal{F}$, denoted by $\nu(\mathcal{F})$, is the maximum number of pairwise disjoint sets in $\mathcal{F}$,
and $\mathcal{F}$ is $d$-cluster-free
if it does not contain $d$ sets with union of size at most $2k$ and empty intersection.
In this paper, we give a lower bound and an upper bound for the maximum size of a $d$-cluster-free
family with a matching number at least $\nu+1$.
In particular, our result of the case $\nu=1$ settles a conjecture of Mammoliti and Britz.
We also introduce a Tur\'{a}n problem in hypergraphs that allows multiple edges,
which may be of independent interest.
\end{abstract}

%%%%%%%%%%%%%%%%%%%%%%%%%%%%%%%%%%%%%%%%%%%%%%%%%%%%%%$
\section{Introduction}
We use $[n]$ to denote the set $\{1,\ldots,n\}$.
For a set $V$ we use $\binom{V}{k}$ to denote the collection of all $k$-subsets of $V$.
A \textit{$d$-cluster} of $k$-sets is a collection of $d$ different $k$-sets $A_1,\ldots,A_d$ such that
\[
|A_1\cup \cdots \cup A_d|\le 2k,\ \text{and}\ |A_1\cap \cdots \cap A_d|=0.
\]
A family $\mathcal{F}\subset \binom{[n]}{k}$ is \textit{$d$-cluster-free} if it does not contain $d$-clusters.
Note that a family is \textit{intersecting} if and only if it is $2$-cluster-free.
The celebrated Erd\H{o}s-Ko-Rado theorem \cite{EKR} states that if $n\ge 2k$ and $\mathcal{F}\subset \binom{[n]}{k}$ is an intersecting family,
then $|\mathcal{F}|\le\binom{n-1}{k-1}$. When $n>2k$, equality holds only if $\mathcal{F}$ is a \emph{star},
i.e. a family   of $k$-sets
that contain a fixed vertex.
In \cite{frankl1976d-wise}, Frankl showed that this theorem still holds for $n\ge dk/(d-1)$ when the intersecting condition
is replaced by the \textit{$d$-wise intersecting} condition, i.e. any $d$ sets of $\mathcal{F}$ have nonempty intersection.

\begin{theorem}[Frankl, \cite{frankl1976d-wise}]\label{thm-Frankl-d-wise}
Let $k\ge d\ge 3$ be fixed and $n\ge dk/(d-1)$. If $\mathcal{F}\subset\binom{[n]}{k}$ is a $d$-wise intersecting
family, then $|\mathcal{F}|\le\binom{n-1}{k-1}$, with equality only if $\mathcal{F}$ is a star.
\end{theorem}

Later, Frankl and F\"{u}redi \cite{frankl1983new} relaxed the intersection condition and  proved that
for every $n\ge k^2+3k$, if $\mathcal{F}\subset \binom{[n]}{k}$ is $3$-cluster-free, then $|\mathcal{F}|\le\binom{n-1}{k-1}$.
Moreover, they conjectured that the lower bound for $n$ can be improved to $3k/2$.
In \cite{mubayi2006erdos}, Mubayi settled their conjecture, and posed the following more general conjecture.

\begin{conjecture}[Mubayi, \cite{mubayi2006erdos}]\label{conj-Mubayi-d-cluster}
Let $k\ge d \ge 3$ and $n\ge dk/(d-1)$.
Suppose that $\mathcal{F}\subset \binom{[n]}{k}$ is $d$-cluster-free.
Then $|\mathcal{F}|\le\binom{n-1}{k-1}$, with equality only if $\mathcal{F}$ is a \emph{star}.
\end{conjecture}

In \cite{mubayi2009set}, Mubayi proved Conjecture \ref{conj-Mubayi-d-cluster} for the case $d=4$ with $n$ sufficiently large.
Later, Mubayi and Ramadurai \cite{mubayi2009set}, and independently, F\"{u}redi and \"{O}zkahya \cite{furedi2011cluster}
proved this conjecture for sufficiently large $n$.
Chen, Liu and Wang \cite{chen2009cluster}
proved this conjecture for the case $d=k$.
In \cite{mammoliti2017}, Mammoliti and Britz showed that this conjecture
is true for \textit{stable} families, i.e. families that are invariant respect to shifting.
Very recently, Currier \cite{GC18} completely solved Conjecture \ref{conj-Mubayi-d-cluster} by proving the following stronger result.

\begin{theorem}[Currier, \cite{GC18}]\label{thm-Currier}
Let $2 \le d \le k \le n/2$.  Furthermore, suppose $\mathcal{F}^* \subset \mathcal{F} \subset {[n] \choose k}$ have the property that any $d$-cluster in $\mathcal{F}$ is contained entirely in $\mathcal{F}^*$. Then
\[
|\mathcal{F}^*| + \frac{n}{k}|\mathcal{F} - \mathcal{F}^*| \le {n \choose k}.
\]
Furthermore, excepting the case where both $d = 2$ and $n = 2k$, equality implies one of the following:
\begin{enumerate}
\item $\mathcal{F}^* = \emptyset$ and $\mathcal{F}$ is a maximum-sized star.
\item $\mathcal{F} = \mathcal{F}^* = {[n] \choose k}$.
\end{enumerate}
\end{theorem}
Note that Theorem \ref{thm-Currier} indeed implies Conjecture \ref{conj-Mubayi-d-cluster}
since $\mathcal{F}$ is $d$-cluster-free if and only if $\mathcal{F}^* = \emptyset$,
and the case $dk/(d-1) \le n < 2k$ has been settled by Theorem \ref{thm-Frankl-d-wise}.

In this paper, we mainly consider a conjecture raised by Mammoliti and Britz.
In \cite{mammoliti2017}, they sharpened Conjecture \ref{conj-Mubayi-d-cluster} further by distinguishing
the two conditions given by Theorem 1.1 and Conjecture \ref{conj-Mubayi-d-cluster}, and considered families
that are $d$-cluster-free but that are not $d$-wise intersecting.
In particular, they posed the following conjecture.

\begin{conjecture}[Mammoliti and Britz, \cite{mammoliti2017}]\label{conj-Mammoliti-Britz}
For $k\ge d\ge 3$ and sufficiently large $n$ every family $\mathcal{F}\subset \binom{[n]}{k}$ that is $d$-cluster-free
but that is not intersecting has size at most $\binom{n-k-1}{k-1}+1$, and
equality holds only if $\mathcal{F}$ is the disjoint union of a $k$-set and a star.
\end{conjecture}

The \textit{matching number} $\nu(\mathcal{F})$ of a family $\mathcal{F}$ is the maximum number of pairwise disjoint sets in $\mathcal{F}$.
Let $f(n,k,d,\nu)$ denote the maximum size of a $d$-cluster-free family $\mathcal{F}\subset \binom{[n]}{k}$ with
a matching number at least $\nu+1$.
Note that by definition $f(n,k,d,0)$ is the maximum size of a $d$-cluster-free $k$-uniform family,
and $f(n,k,d,1)$ is the maximum size of a $k$-uniform family that is $d$-cluster-free but not intersecting.
Conjecture \ref{conj-Mammoliti-Britz} says that $f(n,k,d,1)\le \binom{n-k-1}{k-1}+1$ holds for sufficiently large $n$.

In this paper, we mainly consider the function $f(n,k,d,\nu)$ for $\nu$ fixed and $n$ sufficiently large.
Let $g,h$ be two functions of $n$. Then $f=o(g)$ means that $\lim_{n\to \infty} f/g=0$.
A lower bound and an upper bound for $f(n,k,d,\nu)$ will be given in the remaining part.
The lower bound is given by some constructions, and it is related to the Tur\'{a}n functions on hypergraphs.
On the other hand, the proof of the upper bound is based on a stability theorem proved by Mubayi in \cite{mubayiintersection4}.
So, before stating our results formally, first let us give some definitions.

An $r$-uniform family is also called an $r$-graph.
We use the term \textit{$r$-graph} to emphasize that multiple edges are not allowed in such a hypergraph,
and use the term \textit{$r$-multigraph} to emphasize that multiple edges are allowed in such a hypergraph.
Let $E(\mathcal{G})$ denote the edge set of $\mathcal{G}$,
and let $e(\mathcal{G})$ denote the number of edges in $\mathcal{G}$.
If $\mathcal{G}$ is a hypergraph, then we also use $\mathcal{G}$ to denote the edges set of $\mathcal{G}$.
Suppose that $\mathcal{G}$ is an $r$-multigraph and $E\in \mathcal{G}$ is an edge with multiplicity $\ell$, then $E$ is
counted $\ell$ times in $e(\mathcal{G})$.
Intuitively, one can view $E$ as a set with $\ell$ different colors $c_1,\ldots,c_{\ell}$,
and use $(E,c_i)$ to represent the edge $E$ with color $c_i$.
Pairs $(E,c_i), (E,c_j)$ are considered as different edges in $\mathcal{G}$ if $c_i\neq c_j$.

\begin{dfn}
Let $\mathcal{H}_{v}^{e}$ to be the collection of all $r$-multigraphs on $v$ vertices with $e$ edges.
Let $H_{v}^{e}$ be the collection of $r$-graphs in $\mathcal{H}_{v}^{e}$.
An $r$-multigraph $\mathcal{G}$ is $\mathcal{H}_{v}^{e}$-free if it does not contain any element in $\mathcal{H}_{v}^{e}$ as a subgraph.
An $r$-graph $G$ is $H_{v}^{e}$-free if it  does not contain any element in $H_{v}^{e}$ as a subgraph.
\end{dfn}

Let $EX^{r}(n,\mathcal{H}_{v}^{e})$ denote the maximum number of edges in an $n$-vertex $\mathcal{H}_{v}^{e}$-free $r$-multigraph.
Let $ex^{r}(n,H_{v}^{e})$ denote the maximum number of edges in an $n$-vertex $H_{v}^{e}$-free $r$-graph.
Sometimes we omit the superscript $r$ if there is no cause of any ambiguity.

Let $n,r,t,\lambda$ be integers and $n\ge r\ge t\ge 0$, $\lambda\ge 1$.
A \textit{$t$-$(n,r,\lambda)$-design} is an $r$-graph
$\mathcal{G}$ on $[n]$ such that for every $t$-subset $T$ of $[n]$
there are exactly $\lambda$ members of $\mathcal{G}$ containing $T$.
The existence of certain designs was established by Keevash \cite{keevash2014existence}.

For $\ell\ge 1$ and $r\ge 2$ a \emph{tight $\ell$-path} $P_{\ell}^{r}$ is an $r$-graph with edge set
$\{v_{i}v_{i+1}\ldots v_{i+r-1}: 1\le i\le \ell\}$.
Let $ex\left(n,P_{\ell}^{r}\right)$ denote the maximum number of edges in an $n$-vertex $P_{\ell}^{r}$-free $r$-graph.
Notice that an $r$-graph $\mathcal{G}$ on $[n]$ is $P_{2}^{r}$-free if and only if every $(r-1)$-subset of $[n]$
is contained in at most one edge in $\mathcal{G}$.
Therefore, we have $ex\left(n,P_{2}^{r}\right)\le \frac{1}{r}\binom{n}{r-1}$.
On the other hand, by results in \cite{keevash2014existence}, for infinitely many $n$,
an $(r-1)$-$(n,r,1)$-design exists and, hence, we know that $ex\left(n,P_{2}^{r}\right)\ge \frac{1}{r}\binom{n}{r-1}$
holds for infinitely many $n$.

\medskip

Now we are ready to state our results formally.

\begin{theorem}\label{thm-3-cluster}
There exist two constants $c_1,c_2$ that are  only related to $k,\nu$ and satisfying
\[
\max\left\{(k-1)ex\left(\nu,P_{2}^{3}\right),2(k-1)\lf\frac{\nu}{2}\rf \right\}\le c_1\le c_2\le \frac{k}{3}\binom{\nu}{2}+(k-1)\nu
\]
such that
\[
f(n,k,3,\nu)\ge \binom{n-k\nu-1}{k-1}+\lf\frac{\nu}{2}\rf\binom{n-k\nu-1}{k-3}+c_1\binom{n-k\nu-1}{k-4}+\nu \text{ holds for all }n,
\]
and
\[
f(n,k,3,\nu)\le\binom{n-k\nu-1}{k-1}+\lf\frac{\nu}{2}\rf\binom{n-k\nu-1}{k-3}+(c_2+o(1))\binom{n-k\nu-1}{k-4}+M_3
\]
holds for sufficiently large $n$, where $M_3$ is a constant only related to $k,v$, and $M_3 \le f(k\nu,k,3,\nu-1)$.
\end{theorem}

\begin{theorem}\label{thm-4-cluster}
There exist two constants $c'_1,c'_2\ge k\lf\frac{\nu^2}{4}\rf$ such that
\[
f(n,k,4,\nu) \ge \binom{n-k\nu-1}{k-1}+c'_1\binom{n-k\nu-1}{k-3} \text{ holds for all }n,
\]
and
\[
f(n,k,4,\nu)\le\binom{n-k\nu-1}{k-1}+c'_2\binom{n-k\nu-1}{k-3} \text{ holds for sufficiently large }n.
\]
In particular, if $\nu=1$, then
\[
f(n,k,4,1)\ge \binom{n-k-1}{k-1}+ex(n-k-1,P_{2}^{k-2})+1 \text{ holds for all }n.
\]
\end{theorem}

\begin{theorem}\label{thm-d-cluster}
Suppose that $d\ge 5$. Then
\[
f(n,k,d,\nu)\ge \binom{n-k\nu-1}{k-1}+\nu EX^{k-2}\left(n-k\nu-1,\mathcal{H}_{k-1}^{d-2}\right)+\nu \text{ holds for all }n,
\]
and
\[
f(n,k,d,\nu)\le\binom{n-k\nu-1}{k-1}+(\nu+o(1))EX^{k-2}\left(n-k\nu-1,\mathcal{H}_{k-1}^{d-2}\right)
\]
holds for sufficiently large $n$.
\end{theorem}

For the special case $\nu=1$, we have the following result.

\begin{theorem}\label{thm-3-cluster-exact}
For sufficiently large $n$, we have
\[
f(n,k,3,1)= \binom{n-k-1}{k-1}+1,
\]
with equality only for the disjoint union of a $k$-set and a star.
\end{theorem}

Theorem \ref{thm-3-cluster-exact} shows that Conjecture \ref{conj-Mammoliti-Britz} is true for $d=3$.
However, Theorems \ref{thm-4-cluster} and \ref{thm-d-cluster} imply
that Conjecture \ref{conj-Mammoliti-Britz} is false for $d\ge 4$.

Note that in \cite{mammoliti2017} Mammoliti and Britz also asked for
the maximum size of a $k$-uniform family that is $d$-cluster-free
but that is not $d$-wise intersecting.
Let $g(n,k,d,t)$ denote the maximum size of a $k$-uniform family $\mathcal{F}$ on $[n]$ that is $d$-cluster-free
but not $t$-wise intersecting.
i.e. for all distinct sets  $A_1,\ldots,A_d\in \mathcal{F}$, we have
$
A_1\cap \ldots \cap A_d\neq \emptyset \text{ whenever } |A_1\cup \ldots \cup A_d|\le 2k
$,
but there exist $t$ sets $A'_1,\ldots,A'_t\in \mathcal{F}$ such that $A'_1\cap \ldots \cap A'_t=\emptyset$.
Later, it will be shown that a family $\mathcal{F}$ that is $d$-cluster-free
but not $t$-wise intersecting and of large size is actually not intersecting.
Therefore, we have the following result.

\begin{theorem}\label{thm-d-cluste-d-wise}
The equation $g(n,k,d,t)=f(n,k,d,1)$ holds for sufficiently large $n$.
\end{theorem}

The remaining part of this paper is organized as follows.
First we present some preliminary definitions and lemmas in Section 2.
Since the proofs of these lemmas are basically the same as their original form,
we include most of them in Appendix A.
In Section 3 we will present the proofs of Theorems \ref{thm-3-cluster} and \ref{thm-3-cluster-exact}.
Then, we will present the proof of Theorem \ref{thm-4-cluster} in Section 4,
and the proof of Theorem \ref{thm-d-cluster} will be included in Section 5.
We will prove Theorem \ref{thm-d-cluste-d-wise} in Section 6, and include some remarks in Section 7.
In order to give a lower bound for $f(n,k,d,\nu)$,
we give several constructions of families that are $d$-cluster-free with a matching number at least $\nu +1$.
Since it is quite straightforward to check that these families are $d$-cluster-free with a matching number at least $\nu +1$,
we include this part in Appendix B.
%%%%%%%%%%%%%%%%%%%%%%%%%%%%%%%%%%%%%%%%%%%%%%%%%%%%%

\section{Preliminaries}
Let $S, T$ be two subsets of a set $V$.
Then we use $S - T$ to denote the set $S \setminus T$.
If $T$ contains only one element, say $y$, then sometimes we simply use $y$ instead of $\{y\}$ to represent the set $T$.
Let $\mathcal{G}$ be an $r$-graph (or $r$-multigraph).
We use $V(\mathcal{G})$ to denote the vertex set of $\mathcal{G}$.
The \textit{shadow} $\partial\mathcal{G}$ of $\mathcal{G}$ is defined by
\[
\partial\mathcal{G} = \left\{ A\in \binom{V(\mathcal{G})}{r-1}: \exists B\in \mathcal{G} \text{ such that } A\subset B \right\}.
\]
Let $S$ be a subset of $V(\mathcal{G})$.
We will use $\mathcal{G}[S]$ to denote the induced subgraph of $\mathcal{G}$ on $S$.

For  $r$-graphs, it is well known that the \textit{Tur\'{a}n density}
$\pi\left(H_{v}^{e}\right):=\lim_{n\to\infty}ex\left(n,H_{v}^{e}\right)/\binom{n}{r}$ exists,
and have {the Supersaturation Lemma} (we refer the reader to a detailed survey of hypergraph Tur\'{a}n problems by Keevash \cite{keevash2011hypergraph}).
A similar result is also true for $r$-multigraphs.
\begin{lemma}\label{lemma-multi-Turan-density-exist}
The limit $\lim_{n\to \infty}EX^{r}\left(n,\mathcal{H}_{v}^{e}\right)/\binom{n}{r}$ exists.
\end{lemma}
\begin{proof}
Let $\mathcal{G}$ be an $n$-vertex $\mathcal{H}_{v}^{e}$-free $r$-multigraph with $EX(n,\mathcal{H}_{v}^{e})$ edges.
Choose an $(n-1)$-subset $S$ of $V(\mathcal{G})$ uniformly at random.
For every edge $E\in \mathcal{G}$, the probability that $E$ is contained in $S$ is $(n-r)/n$.
So, the expected number of edges in $S$ is $\left( (n-r)/n \right) EX(n,\mathcal{H}_{v}^{e})$.
Therefore, there exists a set $S$ of size $n-1$ with at least $\left( (n-r)/n \right) EX(n,\mathcal{H}_{v}^{e})$ edges in $\mathcal{G}[S]$.
Since $\mathcal{G}[S]$ is also $\mathcal{H}_{v}^{e}$-free, we therefore have that $\left( (n-r)/n \right) EX(n,\mathcal{H}_{v}^{e})\le EX(n-1,\mathcal{H}_{v}^{e})$.
It follows that
\[
\frac{EX(n,\mathcal{H}_{v}^{e})}{\binom{n}{r}}\le \frac{EX(n-1,\mathcal{H}_{v}^{e})}{\binom{n-1}{r}}.
\]
So $ EX(n,\mathcal{H}_{v}^{e})/\binom{n}{r} $ is non-increasing respect to $n$,
and this implies the existence of the limit
$\lim_{n\to \infty}EX^{r}(n,\mathcal{H}_{v}^{e})/\binom{n}{r}$.
\end{proof}

Define the Tur\'{a}n density $\Pi\left(\mathcal{H}_{v}^{e}\right)$ of $\mathcal{H}_{v}^{e}$ as $\Pi\left(\mathcal{H}_{v}^{e}\right)=\lim_{n\to \infty}EX^{r}\left(n,\mathcal{H}_{v}^{e}\right)/\binom{n}{r}$.
Notice that in the proof of Lemma \ref{lemma-multi-Turan-density-exist},
we showed that $EX^{r}\left(n,\mathcal{H}_{v}^{e}\right)/\binom{n}{r}$
is non-increasing respect to $n$. Therefore, we have
$
EX^{r}\left(n,\mathcal{H}_{v}^{e}\right)\le \left( EX^{r}\left(v,\mathcal{H}_{v}^{e}\right)/\binom{v}{r} \right) \binom{n}{r}< \left( e/\binom{v}{r}\right)\binom{n}{r}
$.
On the other hand, since
every $H_{v}^{e}$-free $r$-graph is also an $\mathcal{H}_{v}^{e}$-free $r$-multigraph,
we have $EX^{r}(n,\mathcal{H}_{v}^{e})\ge ex^{r}(n,H_{v}^{e})$.

\begin{lemma}[Supersaturation]\label{lemma-supersaturation}
For any $\mathcal{H}_{v}^{e}$ and any $a>0$, there exist $b>0$ and $n_0$ such that
any $r$-multigraph $\mathcal{G}$ on $n>n_0$ vertices with at least $\left(\Pi(\mathcal{H}_{v}^{e})+a\right)\binom{n}{r}$ edges
contains at least $b\binom{n}{v}$ copies of elements in $\mathcal{H}_{v}^{e}$.
Moreover, we have $b\ge (a/2)/\binom{M}{v}$, where $M$ is the smallest integer satisfying both $M\ge \max\{r,v\}$
and $EX\left(M,\mathcal{H}_{v}^{e}\right)\le \left(\Pi(\mathcal{H}_{v}^{e})+a/2\right)\binom{M}{r}$.
\end{lemma}

Let $\mathcal{F}\subset \binom{[n]}{k}$ and $x\in[n]$, define
$
\mathcal{F}(x)=\left\{F\in\mathcal{F}: x\in F\right\}
$
and
$
\mathcal{F}(\bar{x})=\left\{F\in\mathcal{F}: x\not\in F\right\}
$.
The following stability theorem for $d$-cluster-free families is an important tool in our proofs.

\begin{theorem}[Stability, \cite{mubayiintersection4}]\label{thm-stability-Mubayi}
Fix $2\le d\le k$. For every $\delta>0$, there exists $\epsilon>0$ and $n_0$ such that
the following holds for all $n>n_0$. Suppose that $\mathcal{F}\subset\binom{[n]}{k}$ is
a $d$-cluster-free family. If $|\mathcal{F}|\ge (1-\epsilon)\binom{n-1}{k-1}$,
then there exists a vertex $x\in[n]$ such that
$
\left|\mathcal{F}(\bar{x})\right|<\delta\binom{n-1}{k-1}
$.
\end{theorem}

Now let $\mathcal{F}$ be a $d$-cluster-free family with a matching number at least
$\nu+1$ and of size exactly $f(n,k,d,\nu)$.
In order to apply Theorem \ref{thm-stability-Mubayi} to $\mathcal{F}$, we need a lower bound for $f(n,k,d,\nu)$.
So, let us give a simple construction of a $d$-cluster-free family $\mathcal{S}_{\nu}$ with a matching number exactly $\nu+1$.

Fix a vertex $y\in[n]$, and choose $\nu$ disjoint sets $C_1,\ldots,C_{\nu}$ from $\binom{[n]-y}{k}$.
Let $J=\bigcup_{i=1}^{\nu}C_i$ and $W=[n]-y-J$.
Let
\[
\mathcal{S}_{\nu} = \left\{\{y\} \cup A: A\in\binom{W}{k-1}\right\}\cup \{C_1,\ldots,C_{\nu}\}.
\]
Note that the size of $\mathcal{S}_{\nu}$ is $\binom{n-k\nu-1}{k-1}+\nu$.
Therefore, we have $f(n,k,d,\nu)\ge \binom{n-k\nu-1}{k-1}+\nu$.

For fixed $\nu$ and $k$ we have $\lim_{n\to \infty}\binom{n-k\nu-1}{k-1}/\binom{n-1}{k-1}=1$.
Choose $\delta>0$ to be sufficiently small, which will be determined later in the proof of Lemma \ref{lemma-Mubayi-m-large-d-cluster},
and let $\epsilon, n_0$ be given by Theorem \ref{thm-stability-Mubayi}.
Let $n$ be sufficiently large so that $n>n_0$ and $\binom{n-k\nu-1}{k-1}>(1-\epsilon)\binom{n-1}{k-1}$.
By Theorem \ref{thm-stability-Mubayi}, there exists a vertex $x\in [n]$ such that $|\mathcal{F}(\bar{x})|<\delta\binom{n-1}{k-1}$.
Since $\mathcal{F}$ contains at least $\nu+1$ pairwise disjoint sets,
we know that $\mathcal{F}(\bar{x})$ contains at least $\nu$ pairwise disjoint sets.
So we can choose $\nu$ pairwise disjoint sets $B_1,\ldots,B_{\nu}$ from $\mathcal{F}(\bar{x})$.
Let $I=\bigcup_{i=1}^{\nu}B_i$ and $U=[n]-x-I$.
Let $m=|\mathcal{F}(\bar{x})|$ and note that $m<\delta\binom{n-1}{k-1}$.
Actually, the following lemmas will show that if $m\ge c\binom{n-1}{k-2}$ holds
for some absolute constant $c>0$,
then there exists a $d$-cluster in $\mathcal{F}$, which contradicts our assumption.

\begin{lemma}[\cite{mubayi2009set}]\label{lemma-mubayi-bipartite-d-cluster}
Fix $2\le d\le k$, $1\le p\le k$, and $k<u_1 \le n/2$ with $n$ sufficiently large.
Suppose that $[n]$ has a partition $U_1\cup U_2$, $u_1=|U_1|$, $u_2=|U_2|$ and $\mathcal{F}$ is a
collection of $k$-sets of $[n]$ such that $|F\cap U_1|=p$ for every $F\in \mathcal{F}$.
If $\mathcal{F}$ contains no $d$-cluster, then $|\mathcal{F}|\le ku_1^{p-1}u_2^{k-p}$.
\end{lemma}

The original form of the next lemma is Claim 1 in \cite{mubayiintersection4}.
Note that it is assumed  in the proof of Claim 1 that the size of $\mathcal{F}$ is at least $\binom{n-1}{k-1}$.
However, in our proof, we can only assume that $|\mathcal{F}|\ge \binom{n-k\nu-1}{k-1}+\nu$.
So we add an extra assumption that $m\ge c\binom{n-1}{k-2}$ holds for some constant $c>0$ in the next lemma,
and the conclusion is also sightly different from that in Claim 1.
\begin{lemma}\label{lemma-Mubayi-3-dijoint-sets}
Suppose that $m\ge c\binom{n-1}{k-2}$ holds for some constant $c>0$.
Then, there are pairwise disjoint
$(k-2)$-sets $S_1,S_2,S_3\subset [n] - x$ such that for each $i$
\[
d_{\mathcal{F}(x)}(S_i):=|\{y\in[n]: \{x,y\}\cup S_i\in \mathcal{F}\}|\ge n-k+1- \frac{\left(k^2/c+2k\right)m}{\binom{n-1}{k-2}}.
\]
\end{lemma}

The proof of the next lemma  appeared in \cite{mubayi2009set} as a part of the proof of its main theorem.
For completeness, we state it formally as a lemma and include its proof in Appendix A.
\begin{lemma}[\cite{mubayi2009set}]\label{lemma-Mubayi-m-large-d-cluster}
Suppose that $m\ge c\binom{n-1}{k-2}$ holds for some constant $c>0$.
Then, there is a $d$-cluster in $\mathcal{F}$.
\end{lemma}

Before presenting our proofs,
we would like to remaind the reader that
in the proof of the upper bound for $f(n,k,d,\nu)$,
we always assume that $n$ is sufficiently large.
Our constructions are obtained from $\mathcal{S}_{\nu}$ by adding some extra $k$-sets.
We will continue using the notations $y, J, W$ and $C_1,\ldots,C_{\nu}$ in the lower bound parts,
and continue using the notations $x, I, U$ and $B_1,\ldots,B_{\nu}$ in the upper bound parts.

%%%%%%%%%%%%%%%%%%%%%%%%%%%%%%%%%%%%%%%%%%%%%%%%%%%%%%%%%%%%%%%%%%%%%%%%%%%%%
\section{Proofs of Theorems \ref{thm-3-cluster} and \ref{thm-3-cluster-exact}}
The proof of Theorem \ref{thm-3-cluster} is consisting of two parts.
In the first part, we present two constructions to give two lower bounds for $f(n,k,3,\nu)$.
In the second part, we prove the upper bound for $f(n,k,3,\nu)$.
\subsection{Lower Bound}
Before presenting our constructions we would like to remind the reader that
the family $S_{\nu}$ is the disjoint union of a star and $\nu$ pairwise disjoint $k$-sets $C_{1},\ldots,C_{\nu}$,
and the definition of $S_{\nu}$ can be found below Theorem \ref{thm-stability-Mubayi} in Section 2.

\noindent$\bullet$ \textbf{First construction for $d=3$}.\\
Choose one vertex $v_i$ from each set $C_i$.
For every $\ell\in \{1,\ldots,\lf\nu/2\rf\}$ let $P_{\ell}=C_{2\ell-1}\cup C_{2\ell}$.
For every $i\in \{2,\ldots,k-1\}$ define
\[
\mathcal{G}_i=\left\{ A \in \bigcup_{\ell=1}^{\lf\nu/2\rf}\binom{P_{\ell}}{i}: \{v_{2\ell-1},v_{2\ell}\}\subset A \text{ for some } \ell\right\}.
\]
Let
\[
\mathcal{L}_1=\mathcal{S}_{\nu} \cup \left(\bigcup_{i=2}^{k-1}\left\{\{y\}\cup A\cup B: B\in\binom{W}{k-1-i} \text{ and } A\in \mathcal{G}_i\right\}\right).
\]
Note that the size of $\mathcal{G}_i$ is $\lf\frac{\nu}{2}\rf\binom{2k-2}{i-2}$ for all $i\in \{2,\ldots,k-1\}$.
Therefore, we have
\[
|\mathcal{L}_1|=\binom{n-k\nu-1}{k-1}+\sum_{i=2}^{k-1}\lf\frac{\nu}{2}\rf\binom{2k-2}{i-2}\binom{n-k\nu-1}{k-1-i}+\nu.
\]
Since $\mathcal{L}_1$ is a $3$-cluster-free family with $\nu(\mathcal{L}_1)=\nu+1$, we therefore have that
\[
f(n,k,3,\nu)\ge \binom{n-k\nu-1}{k-1}+\sum_{i=2}^{k-1}\lf\frac{\nu}{2}\rf\binom{2k-2}{i-2}\binom{n-k\nu-1}{k-1-i}+\nu.
\]

%%%%%%%%%%%%%%%%%%%%%%%%%%%%%%%%%%%%%%%%%%%%%%%%%%%%%%%
\noindent $\bullet$ \textbf{Second construction for $d=3$}.\\
Suppose that $C_i=\{c_1^i,\ldots,c_k^i\}$ for $1 \le i \le \nu$.
Then let $V_j=\{c_j^1,\ldots,c_j^{\nu}\}$ for every $j\in[k]$.
Let $G_1$ be the graph on $V_1$ with edge set $\left\{ \{c_{1}^{2i-1},c_{1}^{2i}\}: 1 \le i \le \lf\nu/2\rf \right\}$.
For every $j \in \{2,\ldots,k\}$ let $\mathcal{G}_j$ be a $P_{2}^{3}$-free $3$-graph on $V_j$ with exactly $ex(\nu,P_{2}^{3})$ edges.
Let
\[
\mathcal{L}'_2=\mathcal{S}_{\nu}\cup\left\{\{y\}\cup A\cup B: B\in\binom{W}{k-3} \text{ and } A\in E(G_1)\right\}.
\]
Then let
\[
\mathcal{L}_2=\mathcal{L}'_2\cup \left(\bigcup_{j=2}^{k}\left\{\{y\}\cup A\cup B: B\in\binom{W}{k-4} \text{ and } A\in\mathcal{G}_j\right\}\right).
\]
It is easy to see that
\[
|\mathcal{L}_2|=\binom{n-k\nu-1}{k-1}+\lf\frac{\nu}{2}\rf\binom{n-k\nu-1}{k-3}+(k-1)ex(\nu,P_{2}^{3})\binom{n-k\nu-1}{k-4}+\nu.
\]
Since $\mathcal{L}_2$ is a $3$-cluster-free family with $\nu(\mathcal{L}_2)=\nu+1$, we therefore have that
\[
f(n,k,3,\nu)\ge \binom{n-k\nu-1}{k-1}+\lf\frac{\nu}{2}\rf\binom{n-k\nu-1}{k-3}+(k-1)ex(\nu,P_{2}^{3})\binom{n-k\nu-1}{k-4}+\nu.
\]
%%%%%%%%%%%%%%%%%%%%%%%%%%%%%%%%%%%%%%%%%%%%%%%%%%

\subsection{Upper Bound}
First we claim that $|F\cap F'|\le k-2$ holds for every $F\in \mathcal{F}(x)$ and every $F'\in \mathcal{F}(\bar{x})$.
Indeed, suppose that there exists an edge $F\in \mathcal{F}(x)$ and an edge $F'\in \mathcal{F}(\bar{x})$ such that $|F\cap F'|=k-1$.
Then for every set $S\in \binom{[n]-x-F'}{k-1}$ we have $\{x\}\cup S\not\in\mathcal{F}$,
since otherwise $\{x\}\cup S, F \text{ and } F'$ would form a $3$-cluster, a contradiction.
So in this case we would have
\[
\begin{split}
&|\mathcal{F}|\le \binom{n-1}{k-1}-\binom{n-k-1}{k-1}+\delta\binom{n-1}{k-1}<\binom{n-k\nu-1}{k-1},
\end{split}
\]
and this contradicts our assumption that $\mathcal{F}$ is of size $f(n,k,3,\nu)$.

Let $M_3$ be the maximum possible number of  sets in $\mathcal{F}$ that are completely contained in $I$,
and it is easy to see that $M_3 \le f(k\nu,k,3,\nu-1)$.
For every subset $C$ of $U$ that of size at most $k-2$ let
\[
\mathcal{F}'(C)=\left\{F-x-C: F\in\mathcal{F}(x)\ \text{and}\ F\cap U=C\right\},
\]
and let
$
\mathcal{F}(C)=\mathcal{F}'(C)-\bigcup_{i=1}^{\nu}\binom{B_i}{k-1-|C|}
$.
For every $j\in\{0,\ldots,k-1\}$ let
$$
\mathcal{F}_j=\left\{F\in\mathcal{F}(x): |F\cap I|=j \right\}.
$$

Intuitively, one can view $\mathcal{F}'(C)$ as the collection of neighbors of $C$ in $I$,
and view $|\mathcal{F}'(C)|$ as the degree of $C$ in $\mathcal{F}(x)$.
Our goal is to give an upper bound for $|\mathcal{F}(x)|$, and this is done by giving an upper
bound for each $|\mathcal{F}'(C)|$.

%%%%%%%%%%%%%%%%%%%%%%%%%%%%%%%%%%%%%%%%%%%%%%%%%%%%%%%%%%%%%%%%%%%%%%%%%%%
\begin{lemma}\label{lemma-3-F(C)-up-bound}
Let $C\in \binom{U}{k-3}$. Then $|\mathcal{F}(C)|\le \lf\frac{\nu}{2}\rf$.
\end{lemma}
\begin{proof}
Let $C\in \binom{U}{k-3}$ and let $G$ denote the graph $\mathcal{F}(C)$.
Note that $G$ is a graph on $I$.
By the definition of $\mathcal{F}(C)$, we know that $B_i$ is an independent set in $G$ for $1\le i \le \nu$.

For every pair $\{i,j\} \subset \{1,\ldots,\nu\}$
let $E(B_i,B_j)$ denote the collection of edges in $G$ that have one endpoint in $B_i$ and the other endpoint in $B_j$,
and let $e(B_i, B_j)$ denote the size of $E(B_i,B_j)$.
First, we claim that $e(B_i,B_j)\le 1$ for every pair $\{i,j\} \subset \{1,\ldots,\nu\}$.
Indeed, suppose that  there are two edges $e_1,e_2\in E(B_i,B_j)$ for some pair $\{i,j\}$.
Assume that $e_1=\{b_1^i,b_1^j\},e_2=\{b_2^i,b_2^j\}$ and $b_1^{i},b_{2}^i\in B_i, b_1^{j},b_{2}^j\in B_j$.
We may assume that $b_1^i\neq b_2^i$, otherwise we consider $b_1^j \text{ and } b_2^j$ instead.
However, the three sets $B_i, \{x,b_1^i,b_1^j\}\cup C \text{ and } \{x,b_2^i,b_2^j\}\cup C$
form a $3$-cluster, a contradiction. Therefore, we have $e(B_i,B_j)\le 1$.
Next, we show that for every $i \in \{1,\ldots,\nu\}$ there is at most one edge that has nonempty intersection with $B_i$.
Indeed, suppose there are two edges $e_1,e_2$ such that $e_1\cap B_i\neq \emptyset$ and $e_2\cap B_i\neq\emptyset$.
Assume that $e_1=\{b_1^i,b_1^j\},e_2=\{b_2^i,b_2^k\}$ and $b_1^{i},b_{2}^i\in B_i, b_1^j\in B_j, b_2^k\in B_k$.
By the argument above, we know that $j\neq k$.
However,
if $b_1^i\neq b_2^i$, then $ B_i, \{x,b_1^i,b_1^j\}\cup C \text{ and } \{x,b_2^i,b_2^k\}\cup C$ form a $3$-cluster, a contradiction.
If $b_1^i= b_2^i$, then $B_j, \{x,b_1^i,b_1^j\}\cup C \text{ and } \{x,b_2^i,b_2^k\}\cup C$ form a $3$-cluster, a contradiction.
So every $B_i$ has nonempty intersection with at most one edge of $G$.
Therefore, we have $|\mathcal{F}(C)|=e(G)\le\lf\nu/2\rf$.
\end{proof}

Assume that $k\ge 4$. Let $C\in \binom{U}{k-4}$ and view $\mathcal{F}(C)$ as a $3$-graph on $I$.
By the definition of $\mathcal{F}(C)$, every $E\in \mathcal{F}(C)$ has nonempty intersection with at least two sets in $\{ B_1,\ldots,B_{\nu}\}$.
We call $E$ a \textit{long edge} if $E$ has nonempty intersection with three sets in $\{ B_1,\ldots,B_{\nu}\}$,
otherwise we call $E$ a \textit{short edge}.
Let $\mathcal{L}_{c}$ be the collection of all long edges in $\mathcal{F}(C)$ and
let  $\mathcal{S}_{c}$ be the collection of all short edges in $\mathcal{F}(C)$.
For every $i\in [\nu]$ let $G_i$ be the graph on $B_i$ with edge set $\partial\mathcal{S}_{c}\cap \binom{B_i}{2}$.
For every pair  $\{i,j\}\subset [\nu]$ let $G_{i,j}$ be the bipartite graph on $B_i\cup B_j$
with edge set $\partial\mathcal{L}_c\cap \binom{B_i\cup B_j}{2}$.

\begin{claim}\label{claim-3-v(G_i)<2}
The matching number of $G_i$ is at most one for every $i \in [\nu]$.
\end{claim}
\begin{proof}
Suppose there are two vertex disjoint edges $e_1,e_2 \text{ in } E(G_i)$ for some $i\in [\nu]$.
By the definition of $E(G_i)$,
there exist two sets $S_1,S_2\in \mathcal{S}_{c}$
such that $S_1\cap B_i=e_1$ and $S_2\cap B_i=e_2$.
However, the three sets $B_i,\{x\}\cup C\cup S_1 \text{ and } \{x\}\cup C\cup S_2$ form a $3$-cluster in $\mathcal{F}$,
a contradiction. Therefore, the matching number of $G_i$ is at most one.
\end{proof}

\begin{claim}\label{claim-3-e=S-cap-Bi}
For every $i\in[\nu]$ and every $e\in E(G_i)$ there is exactly one set $S\in \mathcal{S}_{c}$ such that $S\cap B_i=e$.
\end{claim}
\begin{proof}
Suppose that there exist two vertices $v_1\in B_j$ and $v_2\in B_k$ for some $j,k$ such that $S_1=\{v_1\}\cup e$ and $S_2=\{v_2\}\cup e$
are both contained  in $\mathcal{S}_{c}$.
Here $j\neq i$ and $k\neq i$ but $j,k$ might be the same.
However, the three sets $B_j,\{x\}\cup C\cup S_1 \text{ and }\{x\}\cup C\cup S_2$ form a $3$-cluster in $\mathcal{F}$,
a contradiction. Therefore, there is exactly one set $S\in \mathcal{S}_{c}$ such that $S\cap B_i=e$.
\end{proof}

Claim \ref{claim-3-v(G_i)<2} implies that the size of $E(G_i)$ is at most $k-1$ for every $i \in [\nu]$.
Combining Claim \ref{claim-3-v(G_i)<2} with Claim \ref{claim-3-e=S-cap-Bi}, we obtain that $|\mathcal{S}_c|=\sum_{i=1}^{\nu}|E(G_i)|\le (k-1)v$.
Next, we will give an upper bound for $|\mathcal{L}_c|$.

\begin{claim}\label{claim-3-Delta(Gij)<2}
For every pair $\{i,j\} \subset [\nu]$ every vertex in $G_{i,j}$ has degree at most $1$.
\end{claim}
\begin{proof}
Suppose that there exist two edges $e_1,e_2\in E(G_{i,j})$ for some pair $\{i,j\}\subset [\nu]$ such that $e_1\cap e_2\neq \emptyset$.
Without loss of generality, we may assume that the common endpoint of $e_1,e_2$ lies in $B_i$.
By the definition of $G_{i,j}$, there exist two sets $S_1,S_2\in \mathcal{L}_c$ such that
$e_1\subset S_1$ and $e_2\subset S_2$. However, the three sets $B_j,\{x\}\cup C\cup S_1 \text{ and }\{x\}\cup C\cup S_2$ form a $3$-cluster in $\mathcal{F}$,
a contradiction. Therefore, every vertex in $G_{i,j}$ has degree at most $1$.
\end{proof}

\begin{claim}\label{claim-3-one-S-contain-e-Gij}
For every $e\in E(G_{i,j})$ there is exactly one set $S\in\mathcal{L}_c$ containing $e$.
\end{claim}
\begin{proof}
Suppose there exist two  vertices $v_1\in B_k$ and $v_2\in B_{\ell}$ such that $S_1=\{v_1\}\cup e$
and $S_2=\{v_2\}\cup e$ are both contained in $\mathcal{L}_c$. Here $k,\ell \not\in \{i,j\}$ but $k,\ell$ might be the same.
However, $B_k,\{x\}\cup C\cup S_1 \text{ and }\{x\}\cup C\cup S_2$ form a $3$-cluster in $\mathcal{F}$,
a contradiction. Therefore, there is exactly one set in $\mathcal{L}_c$ that contains $e$.
\end{proof}

Claim \ref{claim-3-Delta(Gij)<2} implies that $|E(G_{i,j})|\le k$ for every pair $\{i,j\} \subset [\nu]$.
Combining Claim \ref{claim-3-Delta(Gij)<2} with Claim \ref{claim-3-one-S-contain-e-Gij},
we obtain that $|\mathcal{L}_c|=\frac{1}{3}\sum_{1\le i<j\le\nu}|E(G_{i,j})|\le \frac{k}{3}\binom{\nu}{2}$.
Since $|\mathcal{F}(C)|=|\mathcal{S}_c|+|\mathcal{L}_c|$, we therefore obtain the following lemma.

\begin{lemma}\label{lemma-3-F(C)-k>=4}
Suppose that $k\ge 4$
and $C\in \binom{U}{k-4}$.
Then
$
|\mathcal{F}(C)|\le \frac{k}{3}\binom{\nu}{2}+(k-1)v
$.
\end{lemma}

%%%%%%%%%%%%%%%%%%%%%%%%%%%%%%%%%%%%%%%%%%%%%%%%%%%%%%%%%%%%%%%%%%%%%%%%%%%
Now we are ready to prove the upper bound for Theorem \ref{thm-3-cluster}.
\begin{proof}
\noindent\textbf{Case 1:} the family $\mathcal{F}(\bar{x}) \text{ is completely contained in } \binom{I}{k}$.\\
For every $j\in[k-2]$ define $\mathcal{B}_j=\bigcup_{i=1}^{\nu}\binom{B_i}{j}$, and let
\[
\mathcal{G}_j=\left\{A\in\binom{U}{j}: \exists B\in \mathcal{B}_{k-1-j} \text{ such that } \{x\}\cup A\cup B\in\mathcal{F}\right\}.
\]
Let $S\in \binom{U}{k-1}$, we say that $S$ is \textit{bad} if it contains an edge $E\in \mathcal{G}_j$ for some $j\in [k-2]$.
Note that if $S$ is bad, then $\{x\}\cup S\not\in\mathcal{F}$,
since otherwise there would be a set $B$ contained in $B_i$ for some  $i$ such that $F=\{x\}\cup E\cup B$
is contained in  $\mathcal{F}$.
However, the three sets $B_i,F \text{ and }\{x\}\cup S$ form a $3$-cluster in $\mathcal{F}$, a contradiction.

For every  $j\in [k-2]$ let $g_j$ denote the size of $\mathcal{G}_j$.
Let $\beta$ denote the number of bad sets in $\binom{U}{k-1}$.
Let $E\in \mathcal{G}_j$.
Then for every $A\in \binom{U-E}{k-1-j}$, we know that $A\cup E$ is a bad set in $\binom{U}{k-1}$.
Therefore, we have
$
\beta\ge \frac{1}{2^{2k}}\sum_{i=1}^{k-2}g_i\binom{|U|-i}{k-1-i}
$.

For every $j\in [k-1]$, we have $|\mathcal{F}_j|\le \binom{|I|}{j}\binom{|U|}{k-1-j}$.
Therefore, we obtain $\sum_{j=4}^{k-1}|\mathcal{F}_j|=o(1)\binom{|U|}{k-4}$.
Let $c'=\frac{k}{3}\binom{\nu}{2}+(k-1)\nu$, by Lemmas \ref{lemma-3-F(C)-up-bound} and \ref{lemma-3-F(C)-k>=4}, we have
\[
\begin{split}
|\mathcal{F}| & =\sum_{i=0}^{k-1}|\mathcal{F}_i|+|\mathcal{F}(\bar{x})|\\ &\le\binom{|U|}{k-1}-\beta+2^k\nu\sum_{i=1}^{k-2}g_i+\lf\frac{\nu}{2}\rf\binom{|U|}{k-3}+(c'+o(1))\binom{|U|}{k-4}+M_3.
\end{split}
\]
For every $j\in[k-2]$, we have $\frac{1}{2^{2k}}\binom{|U|-i}{k-1-j}>2^{k}\nu$.
Therefore, we have $-\beta+2^k\nu\sum_{i=1}^{k-2}g_i\le 0$ and, hence, we obtain
\[
\begin{split}
|\mathcal{F}|&\le \binom{|U|}{k-1}+\lf\frac{\nu}{2}\rf\binom{|U|}{k-3}+(c'+o(1))\binom{|U|}{k-4}+M_3\\
&=\binom{n-k\nu-1}{k-1}+\lf\frac{\nu}{2}\rf\binom{n-k\nu-1}{k-3}+(c'+o(1))\binom{n-k\nu-1}{k-4}+M_3.
\end{split}
\]

\medskip

\noindent\textbf{Case 2:} the family $\mathcal{F}(\bar{x}) \text{ is not completely contained in } \binom{I}{k}$.\\
Then there exists a set $B_{\nu+1}\in\mathcal{F}(\bar{x})$ such that $B_{\nu+1}-I\neq \emptyset$.
Now let $I'=I\cup B_{\nu+1}$ and $U'=[n]-x-I'$.
Let
\[
\mathcal{G}=\left\{E\in\binom{U'}{k-2}: \exists b\in I' \text{ such that } \{x,b\}\cup E\in\mathcal{F}\right\}.
\]
Let $S\in \binom{U'}{k-1}$, we say that $S$ is \textit{bad} if it contains an edge $E\in\mathcal{G}$.
Note that if $S$ is bad, then $\{x\}\cup S\not\in\mathcal{F}$,
since otherwise there would be a vertex $b$ contained in  $B_i$ for some $i$
such that $\{x,b\}\cup E\in\mathcal{F}$.
However, the three sets $B_i, \{x,b\}\cup E\text{ and } \{x\}\cup S$ form a $3$-cluster in $\mathcal{F}$, a contradiction.

Let $g$ denote the size of $\mathcal{G}$ and let $\beta$ denote the number of bad sets in $\binom{U'}{k-1}$.
Let $E\in \mathcal{G}$.
Then for every $v\in U'-E$, we know that $\{v\}\cup E$ is a bad set in $\binom{U'}{k-1}$.
So we have
$
\beta\ge \frac{|U'|-k+2}{k-1}g
$.

Let
$
\mathcal{F}_{\ge 2}=\left\{F\in\mathcal{F}(x): |F\cap I'|\ge 2\right\}
$, and note that
$
|\mathcal{F}_{\ge 2}|\le \sum_{i=2}^{k-1}\binom{|I'|}{i}\binom{|U'|}{k-1-i}< \frac{1}{2}\binom{|U'|}{k-2}
$.
Therefore, we have
\[
\begin{split}
|\mathcal{F}|&=|\mathcal{F}(x)|+|\mathcal{F}(\bar{x})|\le \binom{|U'|}{k-1}-\beta+g|I'|+\frac{1}{2}\binom{|U'|}{k-2}+m\\
&\le \binom{|U'|}{k-1}-\left(\frac{|U'|-k+2}{k-1}-|I'|\right)g+\frac{1}{2}\binom{|U'|}{k-2}+m.
\end{split}
\]
Here we would like to remind the reader that $m$ is the size of $|\mathcal{F}(\bar{x})|$,
which was defined above Lemma \ref{lemma-mubayi-bipartite-d-cluster} in Section 2.

Since $\frac{|U'|-k+2}{k-1}>|I'|$ and $|U'|\le n-k\nu-2$, we therefore have that
\[
\begin{split}
&|\mathcal{F}|\le  \binom{n-k\nu-2}{k-1}+\frac{1}{2}\binom{n-k\nu-2}{k-2}+m= \binom{n-k\nu-1}{k-1}-\frac{1}{2}\binom{n-k\nu-2}{k-2}+m.
\end{split}
\]
By the assumption that $|\mathcal{F}| = f(n,k,3,\nu)$, we obtain $m\ge \frac{1}{2}\binom{n-k\nu-2}{k-2}\ge \frac{1}{4}\binom{n-1}{k-2}$.
However, Lemma \ref{lemma-Mubayi-m-large-d-cluster} implies that $\mathcal{F}$ contains a $3$-cluster, a contradiction.
Therefore, Case 2 is impossible and, hence, we obtain
\[
f(n,k,3,\nu)\le \binom{n-k\nu-1}{k-1}+\lf\frac{\nu}{2}\rf\binom{n-k\nu-1}{k-3}+(c'+o(1))\binom{n-k\nu-1}{k-4}+M_3.
\]
\end{proof}

\subsection{Proof of Theorem \ref{thm-3-cluster-exact}}
\begin{proof}
Let $C$ be a subset of $U$ that of size at most $k-2$.
Since $\nu=1$, every set in $\mathcal{F}'(C)$ is contained in $B_1$ and, hence, we have $\mathcal{F}(C)=\emptyset$.
Note that in the argument above, we already showed that Case 2 is impossible.
Therefore, it suffices to only consider Case 1 and, hence, we obtain
\[
f(n,k,3,1)=|\mathcal{F}|\le \binom{|U|}{k-1}-\beta+2^{k}\sum_{i=1}^{k-2}g_i+1\le \binom{n-k-1}{k-1}+1,
\]
and equality holds only if $g_i=0$ holds for every $i\in[k-2]$,
i.e., $\mathcal{F}$ is the disjoint union of a $k$-set and a star.
\end{proof}
%%%%%%%%%%%%%%%%%%%%%%%%%%%%%%%%%%%%%%%%%%%%%%%%%%%%%%%%%%%%%%%

%%%%%%%%%%%%%%%%%%%%%%%%%%%%%%%%%%%%%%%%%%%%%%%%%%%%%%%%%%%%%%%
\section{Proof of Theorem \ref{thm-4-cluster}}
\subsection{Lower Bound}
\noindent$\bullet$ \textbf{A construction for $\nu = 1$}.\\
Let $k'=k-2$ and $n'=n-k\nu-1$ for short.
Let $\mathcal{G}$ be a $P_{2}^{k'}$-free $k'$-graph on $W$ with exactly $ex(n',P_{2}^{k'})$ edges.
Let $v\in J$ be fixed and define
\[
\mathcal{L}_3=\mathcal{S}_{\nu} \cup \left\{\{y,v\}\cup A: A\in\mathcal{G}\right\}.
\]
It is easy to see that
\[
|\mathcal{L}_3|=\binom{n-k-1}{k-1}+ex(n',P_{2}^{k'})+1.
\]
Since $\mathcal{L}_3$ is $4$-cluster-free and $\nu(\mathcal{L}_3)=2$, we therefore have
\[
f(n,k,4,1)\ge \binom{n-k-1}{k-1}+ex(n',P_{2}^{k'})+1.
\]

%%%%%%%%%%%%%%%%%%%%%%%%%%%%%%%%%%%%%%%%%%%%%%%%%%%%%

\noindent$\bullet$ \textbf{A construction for $\nu \ge 2$}.\\
Let $\mathcal{C}_{\ell}=\{C_1,\ldots,C_{\lf\nu/2\rf}\}$
and $\mathcal{C}_{r}=\{C_{\lf\nu/2\rf+1},\ldots,C_{\nu}\}$.
For every pair $(C_i,C_j)$ with $C_i\in \mathcal{C}_{\ell}$ and $C_j\in \mathcal{C}_{r}$
add  $k$ vertex disjoint edges between $C_i$ and $C_j$,
and let $G$ denote the resulting graph.
Note that the number of edges in $G$ is $k\lf\nu^2/4\rf$.
Let
\[
\mathcal{L}_4=\mathcal{S}_{\nu}\cup \left\{\{y\}\cup e\cup B: B\in\binom{W}{k-3} \text{ and } e\in E(G)\right\}.
\]
Then, it is easy to see that
\[
|\mathcal{L}_4|=\binom{n-k\nu-1}{k-1}+k\lf\frac{\nu^2}{4}\rf\binom{n-k\nu-1}{k-3}+\nu.
\]
Since $\mathcal{L}_4$ is $4$-cluster-free and $\nu(\mathcal{L}_4)=\nu+1$, we therefore have that
\[
f(n,k,4,\nu)\ge \binom{n-k\nu-1}{k-1}+k\lf\frac{\nu^2}{4}\rf\binom{n-k\nu-1}{k-3}+\nu.
\]

\subsection{Upper Bound}
Let $M_4$ be the maximum possible number of sets in  $\mathcal{F}$ that are completely contained in $I$,
and it is easy to see that $M_4 \le f(k\nu,k,4,\nu-1)$.

\begin{proof}
\noindent\textbf{Case 1:} the family $\mathcal{F}(\bar{x}) \text{ is completely contained in } \binom{I}{k}$.\\
For every $i\in [\nu]$ let
\[
\mathcal{G}_i=\left\{A\in\binom{U}{k-2}:\exists b\in B_i\ \text{such that}\ \{x,b\}\cup A\in\mathcal{F} \right\},
\]
and let $g_i$ denote the size of $\mathcal{G}_i$.
Without loss of generality, we may assume that $g_1\ge \cdots \ge g_{\nu}$.
Let $\mathcal{P}_2$ be the collection of all tight $2$-paths in $\mathcal{G}_1$.
Then we have
\[
|\mathcal{P}_2|=\sum_{E\in \binom{U}{k-3}}\binom{d_{\mathcal{G}_1}(E)}{2}\ge \binom{|U|}{k-3}\binom{\sum d_{\mathcal{G}_1}(E)/\binom{|U|}{k-3}}{2}=\frac{(k-2)g_1}{2}\left(\frac{(k-2)g_1}{\binom{|U|}{k-3}}-1\right).
\]
Let $S\in \binom{U}{k-1}$, we say that $S$ is \textit{bad} if it contains at least two sets $E_1,E_2\in\mathcal{G}_i$ for some $i$.
Note that if $S$ is bad, then $\{x\}\cup S\not\in\mathcal{F}$,
since otherwise there would be two vertices $b_1,b_2\in B_i$ such that $\{x,b_1\}\cup E_1,\{x,b_2\}\cup E_2$ are both contained in $\mathcal{F}$.
However, the four sets $B_i, \{x\}\cup S, \{x,b_1\}\cup E_1 \text{ and } \{x,b_2\}\cup E_2$ form a $4$-cluster in $\mathcal{F}$, a contradiction.

Let $\beta$ denote the number of bad sets.
Since every tight $2$-path in $\mathcal{G}_1$ forms a bad set, we have
$
\beta\ge\frac{g_1}{k-1}\left((k-2)g_1/\binom{|U|}{k-3}-1\right)
$.

Let
$
\mathcal{F}_{\ge 2}=\left\{F\in \mathcal{F}(x): |F\cap I|\ge 2\right\}
$.
Then there exists a  constant $c$ such that
$
|\mathcal{F}_{\ge 2}|\le \sum_{i=2}^{k-1}\binom{|I|}{i}\binom{|U|}{k-1-i}\le c\binom{|U|}{k-3}
$.

For every  $E\in \mathcal{G}_i$ there are at most two vertices $b_1,b_2 $ in $ B_i$
such that $\{x,b_1\}\cup E, \{x,b_2\}\cup E\in \mathcal{F}$.
Indeed, suppose there are three vertices $b_1,b_2,b_3\in B_i$ such that
$\{x,b_1\}\cup E, \{x,b_2\}\cup E, \{x,b_3\}\cup E$ are all contained in $\mathcal{F}$.
Then the four sets $\{x,b_1\}\cup E, \{x,b_2\}\cup E, \{x,b_3\}\cup E$ and $B_i$ would form a $4$-cluster in $\mathcal{F}$,
a contradiction.
Therefore, we have
\[
\begin{split}
|\mathcal{F}|&=|\mathcal{F}(x)|+|\mathcal{F}(\bar{x})|\le \binom{|U|}{k-1}-\beta+\sum_{i=1}^{\nu}2g_i+ c\binom{|U|}{k-3}+M_4\\
& \le \binom{n-k\nu-1}{k-1}+2\nu g_1-\frac{g_1}{k-1}\left(\frac{(k-2)g_1}{\binom{|U|}{k-3}}-1\right)+ c\binom{n-k\nu-1}{k-3}+M_4.
\end{split}
\]
Viewing $g_1$ as a variable to obtain that
$
2\nu g_1-\frac{g_1}{k-1}\left((k-2)g_1/\binom{|U|}{k-3}-1\right)\le \frac{(2\nu(k-1)+1)^2}{4(k-1)(k-2)}\binom{|U|}{k-3}
$.
Since $\frac{(2\nu(k-1)+1)^2}{4(k-1)(k-2)}$ is a constant only related to $k$ and $\nu$, we obtain that
\[
|\mathcal{F}|\le \binom{n-k\nu-1}{k-1}+c'_2\binom{n-k\nu-1}{k-3}+M_4,
\]
where $c'_2$ is a constant only related to $k$ and $\nu$.

\medskip

\noindent\textbf{Case 2:} the family $\mathcal{F}(\bar{x}) \text{ is not completely contained in } \binom{I}{k}$.\\
Then there exists a set $B_{\nu+1}\in\mathcal{F}(\bar{x})$ such that $B_{\nu+1}-I\neq \emptyset$.
Now let $I'=I\cup B_{\nu+1}$ and $U'=[n]-x-I'$.
For every $i\in[\nu+1]$ let
\[
\mathcal{G}_i=\left\{A\in\binom{U'}{k-2}:\exists b\in B_i\ \text{such that}\ \{x,b\}\cup A\in\mathcal{F} \right\},
\]
and let $g_i$ denote the size of $\mathcal{G}_i$.
We may assume that $g_1\ge \cdots \ge g_{\nu+1}$.
Let $\mathcal{P}_2$ be the collection of all tight $2$-paths in $\mathcal{G}_1$.
Then
$
|\mathcal{P}_2|\ge\frac{(k-2)g_1}{2}\left((k-2)g_1/\binom{|U'|}{k-3}-1\right)
$.

Let $S\in \binom{U'}{k-1}$, we say that $S$ is \textit{bad} if $S$ contains two edges $E_1,E_2$ in $\mathcal{G}_i$ for some $i$.
Note that if $S$ is bad, then $\{x\}\cup S\not\in\mathcal{F}$.
Let $\beta$ denote the number of bad sets in $\binom{U'}{k-1}$.
Since every tight $2$-path in $\mathcal{G}_1$ forms a bad set,  we have
$
\beta\ge\frac{g_1}{k-1}\left((k-2)g_1/\binom{|U'|}{k-3}-1\right)
$.

Let
$
\mathcal{F}_{\ge 2}=\left\{F\in \mathcal{F}(x): |F\cap I'|\ge 2\right\}
$ and note that there exists a constant $c$ such that
$
|\mathcal{F}_{\ge 2}|\le \sum_{i=2}^{k-1}\binom{|I'|}{i}\binom{|U'|}{k-1-i}\le c\binom{|U'|}{k-3}
$.
Therefore, we have
\[
\begin{split}
|\mathcal{F}| & \le \binom{|U'|}{k-1}-\beta+\sum_{i=1}^{\nu+1}2g_i+ c\binom{|U'|}{k-3}+m\\
& \le \binom{|U'|}{k-1}+2(\nu+1)g_1-\frac{g_1}{k-1}\left(\frac{(k-2)g_1}{\binom{|U'|}{k-3}}-1\right)+ c\binom{|U'|}{k-3}+m.
\end{split}
\]
Since
$
2(\nu+1)g_1-\frac{g_1}{k-1}\left((k-2)g_1/\binom{|U'|}{k-3}-1\right)\le \frac{(2(\nu+1)(k-1)+1)^2}{4(k-1)(k-2)}\binom{|U'|}{k-3}
$,
there exists a constant $c'$ such that
\[
\begin{split}
&|\mathcal{F}|\le \binom{|U'|}{k-1}+c'\binom{|U'|}{k-3}+m\le \binom{n-k\nu-1}{k-1}-\binom{n-k\nu-2}{k-2}+c'\binom{n-k\nu-2}{k-3}+m.
\end{split}
\]
By the assumption that $|\mathcal{F}| = f(n,k,4,\nu)$,
we have $m>\binom{n-k\nu-2}{k-2}-c'\binom{n-k\nu-2}{k-3}\ge \frac{1}{2}\binom{n-1}{k-2}$.
However, Lemma \ref{lemma-Mubayi-m-large-d-cluster} implies that  $\mathcal{F}$ contains a $4$-cluster, a contradiction.
Therefore, Case $2$ is impossible and, hence, there exists a constant $c_2$ such that
\[
f(n,k,4,\nu)\le \binom{n-k\nu-1}{k-1}+c_2\binom{n-k\nu-1}{k-3}.
\]
\end{proof}
%%%%%%%%%%%%%%%%%%%%%%%%%%%%%%%%%%%%%%%%%%%%%%%%%%%%%%%%%%%%%%%

%%%%%%%%%%%%%%%%%%%%%%%%%%%%%%%%%%%%%%%%%%%%%%%%%%%%%%%%%%%%%%%
\section{Proof of Theorem \ref{thm-d-cluster}}
Let $k'=k-2$ and $n'=n-k\nu-1$ for short.
\subsection{Lower Bound}
Let $\mathcal{G}$ be an $n'$-vertex $\mathcal{H}_{k-1}^{d-2}$-free $k'$-multigraph on $W$ with
exactly $EX^{k'}\left(n',\mathcal{H}_{k-1}^{d-2}\right)$ edges.
Let $E\in \mathcal{G}$ be an edge of multiplicity $\ell$.
For every $i\in[\nu]$ choose $\ell$ distinct vertices $c_1^i,\ldots,c_{\ell}^i$
from $C_i$ and add $\{y,c_1^i\}\cup E,\ldots,\{y,c_l^i\}\cup E$ into $\mathcal{S}_{\nu}$.
Let $\mathcal{L}_5$ denote the resulting family.
It is easy to see that
\[
|\mathcal{L}_5|=\binom{n-k\nu-1}{k-1}+\nu EX^{k'}\left(n',\mathcal{H}_{k-1}^{d-2}\right)+\nu.
\]
When $d\ge 5$, every $k'$-graph in $H_{k-1}^{d-2}$ is \textit{nondegenerate}, i.e. the Tur\'{a}n density $\pi\left( H_{k-1}^{d-2}\right)$ of $H_{k-1}^{d-2}$ is not $0$ (the reader may refer to \cite{keevash2011hypergraph} for more details),
we therefore have that
$
\Pi\left(\mathcal{H}_{k-1}^{d-2}\right)\ge \pi\left(H_{k-1}^{d-2}\right)\ge \frac{(k-2)!}{(k-2)^{k-2}}
$.
Since $\mathcal{L}_5$ is a $d$-cluster-free family with $\nu(\mathcal{L}_5)=\nu+1$, we therefore have that
\[
f(n,k,d,\nu)\ge \binom{n-k\nu-1}{k-1}+\nu EX^{k'}\left(n',\mathcal{H}_{k-1}^{d-2}\right)+\nu.
\]

\subsection{Upper Bound}
Let $M_d$ be the maximum possible number of sets in $\mathcal{F}$ that are completely contained in $I$,
and it is easy to see that $M_d \le f(k\nu,k,d,\nu-1)$.

\begin{proof}
\noindent\textbf{Case 1:} the family $\mathcal{F}(\bar{x}) \text{ is completely contained in } \binom{I}{k}$.\\
For every $i\in[\nu]$ define the $k'$-multigraph $\mathcal{G}_{i}$ on $U$ as
\[
\mathcal{G}_i=\left\{E\in\binom{U}{k-2}:\exists b\in B_i\ \text{such that}\ \{x,b\}\cup E\in\mathcal{F} \right\}.
\]
Let $E\in \mathcal{G}_i$.
Then the multiplicity of $E$ is the number of vertices $b$ in $B_i$ such that $\{x,b\}\cup E\in\mathcal{F}$.
For every $i\in[\nu]$ let $g_i$ denote the number of edges in $\mathcal{G}_i$.
Without loss of generality, we may assume that $g_1\ge \cdots \ge g_{\nu}$.

Let $S\in  \binom{U}{k-1}$, we say that $S$ is \textit{bad} if $\mathcal{G}_i[S]\in\mathcal{H}_{k-1}^{d-2}$ holds for some $i$.
Note that if $S$ is bad, then $\{x\}\cup S\not\in\mathcal{F}$,
since otherwise there would be $d-2$ edges $E_1,\ldots,E_{d-2}$ in $\mathcal{G}_i$ for some $i$ such that they are all contained in $S$.
By the definition of $\mathcal{G}_{i}$, there exist $d-2$ vertices  $b_1,\ldots,b_{d-2}\in B_i$ such that $\{x,b_1\}\cup E_1,\ldots,\{x,b_{d-2}\}\cup E_{d-2}$
are all contained in $\mathcal{F}$.
However, the $d$ sets $B_i, \{x\}\cup S,  \{x,b_1\}\cup E_1,\ldots,\{x,b_{d-2}\}\cup E_{d-2}$ form a $d$-cluster in $\mathcal{F}$, a contradiction.

Let
$
\mathcal{F}_{\ge 2}=\left\{F\in \mathcal{F}(x): |F\cap I|\ge 2\right\}
$.
Then there exists a constant $c$ such that
$
|\mathcal{F}_{\ge 2}|\le \sum_{i=2}^{k-1}\binom{|I|}{i}\binom{|U|}{k-1-i}\le c\binom{|U|}{k-3}
$.
Therefore, we have
\[
\begin{split}
&|\mathcal{F}|\le \binom{|U|}{k-1}-\beta+\sum_{i=1}^{\nu}g_i+ c\binom{|U|}{k-3}+M_d.
\end{split}
\]
If $g_1\le (1+o(1))EX^{k'}\left(n',\mathcal{H}_{k-1}^{d-2}\right)$, then we are done.
Therefore, we may assume that $g_1=(1+a)EX^{k-2}\left(n',\mathcal{H}_{k-1}^{d-2}\right)$ with
 $a\ge 2\sigma$ holds for some absolute constant $\sigma>0$.
By Lemma \ref{lemma-supersaturation}, the graph $\mathcal{G}_1$ contains  at least $\frac{a/2}{\binom{N}{k-1}}\binom{|U|}{k-1}$
copies of elements in $\mathcal{H}_{k-1}^{d-2}$, where
$N$ is the smallest integer satisfying both $EX^{k'}\left(N,\mathcal{H}_{k-1}^{d-2}\right)\le (1+\sigma)\Pi\left(\mathcal{H}_{k-1}^{d-2}\right)\binom{N}{k-2}$ and $N\ge k-1$.

Let $\beta$ denote the number of bad sets.
Since every copy of element in $\mathcal{H}_{k-1}^{d-2}$ forms a bad set in $\binom{U}{k-1}$,
we therefore have
\[
\beta\ge \frac{1}{(k-1)^{d-2}}\frac{a}{2\binom{N}{k-1}}\binom{|U|}{k-1}=:ac'\binom{|U|}{k-1},
\]
where $c'=\frac{1}{2(k-1)^{d-2}\binom{N}{k-1}}>0$ is a constant.
Therefore, the size of $\mathcal{F}$ satisfies
\[
\begin{split}
&|\mathcal{F}|\le \binom{|U|}{k-1}-ac'\binom{|U|}{k-1}+\nu (1+a)EX^{k'}\left(n',\mathcal{H}_{k-1}^{d-2}\right)+ c\binom{|U|}{k-3}+M_d.
\end{split}
\]
Since $EX^{k'}\left(n',\mathcal{H}_{k-1}^{d-2}\right)\le \frac{d-2}{k-1}\binom{|U|}{k-2}$,
we obtain that $c'\binom{|U|}{k-1}>\nu EX^{k'}\left(n',\mathcal{H}_{k-1}^{d-2}\right)$ and, hence, we have
\[
\begin{split}
&|\mathcal{F}|\le \binom{|U|}{k-1}+\nu (1+o(1))EX^{k'}\left(n',\mathcal{H}_{k-1}^{d-2}\right).
\end{split}
\]

\medskip

\noindent\textbf{Case 2:} the family $\mathcal{F}(\bar{x}) \text{ is not completely contained in } \binom{I}{k}$.\\
Then there exists a set $B_{\nu+1}\in\mathcal{F}(\bar{x})$ such that $B_{\nu+1}-I\neq\emptyset$.
Now let $I'=I\cup B_{\nu+1}$ and let $U'=[n]-x-I$.
For every $i\in[\nu+1]$ define the $k'$-multigraph $\mathcal{G}_i$ on $U'$ as
\[
\mathcal{G}_i=\left\{E\in\binom{U'}{k-2}:\exists b\in B_i\ \text{such that}\ \{x,b\}\cup E\in\mathcal{F} \right\}.
\]
Let $E\in \mathcal{G}_i$.
Then the multiplicity of $E$ is the number of vertices $b$ in $B_i$ such that $\{x,b\}\cup E\in\mathcal{F}$.
For every $i\in[\nu+1]$ let $g_i$ denote the number of edges in $\mathcal{G}_i$.
Without loss of generality, we may assume that $g_1\ge \cdots \ge g_{\nu+1}$.
Let $S\in \binom{U'}{k-1}$, we say that $S$ is \textit{bad} if $\mathcal{G}_i[S]\in\mathcal{H}_{k-1}^{d-2}$ holds for some $i$.
Note that if $S$ is bad, then $\{x\}\cup S\not\in\mathcal{F}$.
Let $\beta$ denote the number of bad sets.

Let
$
\mathcal{F}_{\ge 2}=\left\{F\in \mathcal{F}(x): |F\cap I'|\ge 2\right\}
$.
Then there exists a constant $c$ such that
\[
|\mathcal{F}_{\ge 2}|\le \sum_{i=2}^{k-1}\binom{|I'|}{i}\binom{|U'|}{k-1-i}\le c\binom{|U'|}{k-3}.
\]
Therefore, we have
\[
\begin{split}
|\mathcal{F}|&\le \binom{|U'|}{k-1}-\beta+\sum_{i=1}^{\nu+1}g_i+ c\binom{|U'|}{k-3}+m\\
&\le \binom{n-k\nu-2}{k-1}-\beta+\sum_{i=1}^{\nu+1}g_i+ c\binom{n-k\nu-2}{k-3}+m\\
&=\binom{n-k\nu-1}{k-1}-\binom{n-k\nu-2}{k-2}-\beta+\sum_{i=1}^{\nu+1}g_i+ c\binom{n-k\nu-2}{k-3}+m.
\end{split}
\]
If $g_1\le (1+o(1))EX^{k-2}\left(|U'|,\mathcal{H}_{k-1}^{d-2}\right)\le (1+o(1))\frac{d-2}{k-1}\binom{n-k\nu-2}{k-2}$,
then
\[
\begin{split}
&|\mathcal{F}|\le \binom{n-k\nu-1}{k-1}+\nu EX^{k-2}\left(n',\mathcal{H}_{k-1}^{d-2}\right)+m-\frac{k-d+1}{2(k-1)}\binom{n-1}{k-2}.
\end{split}
\]
By the assumption that $|\mathcal{F}| = f(n,k,d,\nu)$, we have
$
m>\frac{k-d+1}{4(k-1)}\binom{n-1}{k-2}
$.
However, Lemma \ref{lemma-Mubayi-m-large-d-cluster} implies that $\mathcal{F}$ contains a $d$-cluster, a contradiction.
Therefore, we may assume that $g_1=(1+a)EX^{k'}\left(|U'|,\mathcal{H}_{k-1}^{d-2}\right)$
with $a\ge 2\sigma$ holds for some absolute constant $\sigma>0$.
Lemma 2.2 implies that $\mathcal{G}_1$ contains at least $\frac{a/2}{\binom{N}{k-1}}\binom{|U'|}{k-1}$
copies of elements in $\mathcal{H}_{k-1}^{d-2}$.
Therefore, we have
\[
\beta\ge \frac{1}{(k-1)^{d-2}}\frac{a}{2\binom{N}{k-1}}\binom{|U'|}{k-1}=:ac'\binom{|U'|}{k-1},
\]
where $c'=\frac{1}{2(k-1)^{d-2}\binom{N}{k-1}}>0$ is a constant.
So the size of $\mathcal{F}$ satisfies
\[
\begin{split}
&|\mathcal{F}|\le \binom{|U'|}{k-1}-ac'\binom{|U'|}{k-1}+(\nu+1)(1+a)EX^{k'}(|U'|,\mathcal{H}_{k-1}^{d-2})+ c\binom{|U'|}{k-3}+m.
\end{split}
\]
Since $c'\binom{|U|}{k-1}>(\nu+1)EX^{k'}\left(n',\mathcal{H}_{k-1}^{d-2}\right)$,
we therefore have that
\[
\begin{split}
|\mathcal{F}| & \le \binom{|U'|}{k-1}+(\nu+1)(1+o(1))EX^{k'}\left(|U'|,\mathcal{H}_{k-1}^{d-2}\right)+m\\
&\le \binom{n-k\nu-1}{k-1}-\binom{n-k\nu-2}{k-2}+(\nu+1)(1+o(1))EX^{k'}\left(|U'|,\mathcal{H}_{k-1}^{d-2}\right)+m\\
&\le \binom{n-k\nu-1}{k-1}+\nu EX^{k'}\left(|U'|,\mathcal{H}_{k-1}^{d-2}\right)+m-\frac{k-d+1}{2(k-1)}\binom{n-1}{k-2}.
\end{split}
\]
By the assumption that $|\mathcal{F}| = f(n,k,d,\nu)$, we have
$
m>\frac{k-d+1}{4(k-1)}\binom{n-1}{k-2}
$.
However, Lemma \ref{lemma-Mubayi-m-large-d-cluster} implies that $\mathcal{F}$ contains a $d$-cluster, a contradiction.
Therefore, we have
\[
f(n,k,d,\nu)\le \binom{n-k\nu-1}{k-1}+\nu(1+o(1))EX^{k-2}\left(n-k\nu-1,\mathcal{H}_{k-1}^{d-2}\right).
\]
\end{proof}

%%%%%%%%%%%%%%%%%%%%%%%%%%%%%%%%%%%%%%%%%%%%%%%%%%%%%%%%%%%%
\section{Proof of Theorem \ref{thm-d-cluste-d-wise}}
\begin{proof}
Let $\mathcal{K}\subset \binom{[n]}{k}$ be a family that is $d$-cluster-free but not $t$-wise intersecting and of size $g(n,k,d,t)$.
Notice that a family that is not intersecting is also not $t$-wise intersecting.
Therefore, we have  $g(n,k,d,t)\ge f(n,k,d,1)> \binom{n-k-1}{k-1}$.

Now choose $\delta'>0$ to be sufficiently small such that $\delta'< 2\binom{n-k-1}{k-1}/\binom{n-1}{k-1}-1$ holds for sufficiently large $n$,
and let $\epsilon', n_0'$ be given by Theorem \ref{thm-stability-Mubayi}.
Let $n$ be sufficiently large such that $n>n_0'$ and  $\binom{n-k-1}{k-1}>(1-\epsilon')\binom{n-1}{k-1}$.
By Theorem \ref{thm-stability-Mubayi}, there exists $z\in [n]$ such that $|\mathcal{K}(\bar{z})|<\delta' \binom{n-1}{k-1}$.

Notice that  $\mathcal{K}(\bar{z})$ is nonempty, since otherwise every set in $\mathcal{K}$ would contain $z$,
and this contradicts our assumption that $\mathcal{K}$ is not $t$-wise intersecting.
So, let $D$ be a set in $\mathcal{K}(\bar{z})$ and consider the family $\mathcal{K}(z)$.
We claim that there exists a set $E\in \mathcal{K}(z)$ that is disjoint from $D$.
Indeed, suppose that every set in $\mathcal{K}(z)$ has nonempty intersection with $D$.
Then the size of $\mathcal{K}(z)$ is at most $\binom{n-1}{k-1}-\binom{n-k-1}{k-1}$,
and, hence, we have
\[
|\mathcal{K}|\le \binom{n-1}{k-1}-\binom{n-k-1}{k-1}+\delta' \binom{n-1}{k-1}<\binom{n-k-1}{k-1},
\]
a contradiction.
Therefore, there exists a set $E\in \mathcal{K}(z)$ that is disjoint from $D$.
However, this implies that $\mathcal{K}$ is not intersecting and, hence, we have $g(n,k,d,t)\le f(n,k,d,1)$.
Therefore, the equation $g(n,k,d,t)= f(n,k,d,1)$ holds for sufficiently large $n$.
\end{proof}
%%%%%%%%%%%%%%%%%%%%%%%%%%%%%%%%%%%%%%%%%%%%%%%%%%%%%%%%%%%%%%%%%%%%%%%%

\section{Concluding Remarks}
In Section 3 we give two constructions for the lower bounds for $f(n,k,3,\nu)$.
The first construction shows that
\[
f(n,k,3,\nu)\ge \binom{n-k\nu-1}{k-1}+\sum_{i=2}^{k-1}\lf\frac{\nu}{2}\rf\binom{2k-2}{i-2}\binom{n-k\nu-1}{k-1-i}+\nu,
\]
while the second construction shows that
\[
f(n,k,3,\nu)\ge \binom{n-k\nu-1}{k-1}+\lf\frac{\nu}{2}\rf\binom{n-k\nu-1}{k-3}+(k-1)ex(\nu,P_{2}^{3})\binom{n-k\nu-1}{k-4}+\nu.
\]
Since $ex(\nu,P_{2}^{3}) \ge \binom{\nu}{2}/3$ holds for infinitely many $\nu$,
the second construction is better than the first one for large $\nu$.
However, when $\nu$ is small, say smaller than $7$, then the first construction is better.
So determining the extremal families for $f(n,k,3,\nu)$ seems very complicated in general.

The author has written another paper \cite{XL19STRUCTURE} concerning the structures of conditionally intersecting families,
and gave a second proof for Theorem \ref{thm-3-cluster-exact}.
However, the method we used in the second proof is completely different from the method we used here.
In the second proof our main tool is a structural theorem for $3$-cluster-free families,
and, moreover, in the second proof we showed that $f(n,k,3,\nu) = \binom{n-k\nu-1}{k-1}+1$ holds for all $n \ge 3k\binom{2k}{k}$.

%%%%%%%%%%%%%%%%%%%%%%%%%%%%%%%%%%%%%%%%%%%%%%%%%%%%%%%%%%%
\section{Acknowledgement}
We are very grateful to Dhruv Mubayi for his guidance, expertise, fruitful discussions that greatly assisted this research,
and suggestions that greatly improved the presentation of this paper.
We are also very grateful to the referees for their very careful reading of the manuscript and many helpful suggestions.

%%%%%%%%%%%%%%%%%%%%%%%%%%%%%%%%%%%%%%%%%%%%%%%%%%%%%%%%%%%%%
%%%%%%%%%%%%%%%%%%%%%%%%%%%%%%%%%%%%%%%%%%%%%%%%%%%%%%%%%%%%%
%%%%%%%%%%%%%%%%%%%%%%%%%%%%%%%%%%%%%%%%%%%%%%%%%%%%%%%%%%%%%
%%%%%%%%%%%%%%%%%%%%%%%%%%%%%%%%%%%%%%%%%%%%%%%%%%%%%%%%%%%%%
%%%%%%%%%%%%%%%%%%%%%%%%%%%%%%%%%%%%%%%%%%%%%%%%%%%%%%%%%%%%%
\section{Appendix A}
\subsection{Proof of Lemma \ref{lemma-supersaturation}}
\begin{proof}
Fix $M\ge \max\{r,v\}$ such that $EX(M,\mathcal{H}_{v}^{e})\le (\Pi(\mathcal{H}_{v}^{e})+a/2)\binom{M}{r}$.
Then there must be at least $\left(a/2\right)\binom{n}{M}$ $M$-sets $S\subset V(\mathcal{G})$
inducing an $r$-graph $\mathcal{G}[S]$ with $e(\mathcal{G}[S])>(\Pi(\mathcal{H}_{v}^{e})+a/2)\binom{M}{r}$.
Otherwise, we would have
\[
\begin{split}
&\sum_{S\in \binom{V(\mathcal{G})}{M}}e(\mathcal{G}[S])\le \binom{n}{M}\left(\Pi(\mathcal{H}_{v}^{e})+\frac{a}{2}\right)\binom{M}{r}+\frac{a}{2}\binom{n}{M}\binom{M}{r}=(\Pi(\mathcal{H}_{v}^{e})+a)\binom{n}{M}\binom{M}{r}.
\end{split}
\]
However, we also have
\[
\begin{split}
&\sum_{S\in \binom{V(\mathcal{G})}{M}}e(\mathcal{G}[S])=\binom{n-r}{M-r}e(\mathcal{G})>\binom{n-r}{M-r}\left(\Pi(\mathcal{H}_{v}^{e})+a\right)\binom{n}{r}=(\Pi(\mathcal{H}_{v}^{e})+a)\binom{n}{M}\binom{M}{r},
\end{split}
\]
a contradiction.
By the choice of $M$, every $M$-set $S$ of $V(\mathcal{G})$ contains a copy of an element in $\mathcal{H}_{v}^{e}$.
So the number of copies of elements in $\mathcal{H}_{v}^{e}$ is at least
$\frac{a/2\binom{n}{M}}{\binom{n-v}{M-v}}=\frac{a/2}{\binom{M}{v}}\binom{n}{v}$.
So $b$ is at least $(a/2)/\binom{M}{v}$.
\end{proof}

\subsection{Proof of Lemma \ref{lemma-Mubayi-3-dijoint-sets}}
\begin{proof}
Let $t$ be the number of $(k-2)$-sets $T\subset [n]-x$ satisfying
\[
d_{\mathcal{F}(x)}(T)\ge n-k+1-\frac{\left(k^2/c+2k\right)m}{\binom{n-1}{k-2}}.
\]
Then
\[
\begin{split}
(k-1)|\mathcal{F}(x)|&=\sum_{T'\in\binom{[n]-x}{k-2}}d_{\mathcal{F}(x)}(T')\\
&\le t(n-k+1)+\left(\binom{n-1}{k-2}-t\right)\left(n-k+1-\frac{\left(k^2/c+2k\right)m}{\binom{n-1}{k-2}}\right),
\end{split}
\]
which implies that
\[
\begin{split}
\frac{\left(k^2/c+2k\right)m}{\binom{n-1}{k-2}}t & \ge (k-1)|\mathcal{F}(x)|-\binom{n-1}{k-2}\left(n-k+1-\frac{\left(k^2/c+2k\right)m}{\binom{n-1}{k-2}}\right)\\
&\ge (k-1)\left(\binom{n-k-1}{k-1}-m\right)-\binom{n-1}{k-2}\left(n-k+1-\frac{\left(k^2/c+2k\right)m}{\binom{n-1}{k-2}}\right)\\
&=(k-1)\left(\binom{n-k-1}{k-1}-\binom{n-1}{k-1}\right)+\left(k^2/c+2k-k+1\right)m\\
&\ge \left(k^2/c+k+1\right)m-(k-1)(k+1)\binom{n-k-1}{k-2}\ge km.
\end{split}
\]
Here we used the fact that $\binom{n-1}{k-1}-\binom{n-k-1}{k-1}\le (k+1)\binom{n-k-1}{k-2}$ holds for sufficiently large $n$, and $m\ge c\binom{n-1}{k-2}$.
From the inequality above, we obtain
\[
t\ge \frac{k}{k^2/c+2k}\binom{n-1}{k-2}=\frac{1}{k/c+2}\binom{n-1}{k-2}.
\]
Now let us consider the family of all $(k-2)$-sets described above, and let $T_1,\ldots,T_{\ell}$ be a maximum matching
in this family. Since any other set has non-empty intersection with $\bigcup_{i\in[\ell]}T_i$, we have
$t\le \ell(k-2)\binom{n-1}{k-3}$. So we obtain $\ell\ge \frac{1}{(k-2)\left(k/c+2\right)}\binom{n-1}{k-2}/\binom{n-1}{k-3}$.
When $n$ is sufficiently large, $\ell\ge 3$ and this completes the proof.
\end{proof}

\subsection{Proof of Lemma \ref{lemma-Mubayi-m-large-d-cluster}}
\begin{proof}
By Lemma \ref{lemma-Mubayi-3-dijoint-sets}, there exist three disjoint $(k-2)$-sets $S_1,S_2,S_3\subset [n]\setminus\{x\}$
such that for each $i$
\[
d_{\mathcal{F}(x)}(S_i)\ge n-k+1-\frac{\left(k^2/c+2k\right)m}{\binom{n-1}{k-2}}.
\]
Therefore, for each $i$
\[
|\{y\in[n]:\{x,y\}\cup S_i\not\in \mathcal{F}\}|<k+\frac{\left(k^2/c+2k\right)m}{\binom{n-1}{k-2}}.
\]
Let
$
B=\{y\in[n]:\{x,y\}\cup S_i\not\in \mathcal{F}\ \text{for\ some}\ i\in[3]\}
$.
Then, we have $|B|\le 3k+\frac{\left(3k^2/c+6k\right)m}{\binom{n-1}{k-2}}$.
By adding vertices into $B$, we may assume that
\[
|B|= 3k+\frac{\left(3k^2/c+6k\right)m}{\binom{n-1}{k-2}}.
\]
Since $m\ge c\binom{n-1}{k-2}$ for some constant $c>0$, we have
\[
\begin{split}
|B| & = 3k+\frac{\left(3k^2/c+6k\right)m}{\binom{n-1}{k-2}}\le\frac{\left(6k^2/c+6k\right)m}{\binom{n-1}{k-2}}\\
&\le \frac{\left(6k^2/c+6k\right)\delta\binom{n-1}{k-1}}{\binom{n-1}{k-2}}\le \left(6k^2/c+6k\right)\delta n\le\frac{n-1}{2}.
\end{split}
\]
For each $i\in\{0,1,\ldots,k\}$, define
\[
\mathcal{T}_i=\{T\in\mathcal{F}(\bar{x}):|T\cap B|=i\}.
\]
Note that $\bigcup_{i=0}^{k}\mathcal{T}_i$ is a partition of $\mathcal{F}(\bar{x})$.
First we show that $\mathcal{T}_0=\mathcal{T}_1=\mathcal{T}_2=\emptyset$.
Our first observation is that by definition $S_i\subset B$ for all $i\in[3]$.
If $S\in \mathcal{T}_0\cup \mathcal{T}_1\cup \mathcal{T}_2$, then there is
an $i$ for which $S_i\cup S=\emptyset$.
Choose $d-2\le k-2$ elements $y_1,\ldots,y_{d-2}\in S\setminus B$ and $y\in [n]-x-B-S$.
Now the $d-2$ sets $\{x,y_j\}\cup S_i$ for any $j\in[d-2]$, together with $S$ and $\{x,y\}\cup S_i$ form a $d$-cluster,
a contradiction. Therefore, $\mathcal{T}_0=\mathcal{T}_1=\mathcal{T}_2=\emptyset$.
So, $\mathcal{F}(\bar{x})=\bigcup_{i=3}^{k}\mathcal{T}_i$.
We may assume that
$|\mathcal{T}_p|\ge m/(k-2)$ for some $3\le p\le k$.
Applying Lemma \ref{lemma-mubayi-bipartite-d-cluster} with $U_1=B, U_2=[n]-x-B$ and $u_1=|U_1|, u_2=|U_2|$,
we obtain
\[
\frac{m}{k-2}\le |\mathcal{T}_p|\le ku_1^{p-1}u_2^{k-p}\le k\left(\frac{\left(6k^2/c+6k\right)m}{\binom{n-1}{k-2}}\right)^{p-1}n^{k-p}.
\]
Simplifying the inequality above, we obtain
\[
m^{p-2}\ge \frac{\binom{n-1}{k-2}^{p-1}}{(k-2)k\left(6k^2/c+6k\right)^{p-1}n^{k-p}}.
\]
Since $m\le \delta\binom{n-1}{k-1}\le \delta n\binom{n-1}{k-2}$, we know that
\[
\begin{split}
\delta\ge \delta^{p-2} & \ge \frac{m^{p-2}}{n^{p-2}\binom{n-1}{k-2}^{p-2}}\ge \frac{\binom{n-1}{k-2}}{(k-2)k\left(6k^2/c+6k\right)^{p-1}n^{k-2}}\\
&\ge \frac{\left(\frac{n-k}{n}\right)^{k-2}}{(k-2)!(k-2)k\left(6k^2/c+6k\right)^{p-1}}\ge \frac{1}{2(k-2)!(k-2)k\left(6k^2/c+6k\right)^{p-1}}
\end{split}
\]
holds for sufficiently large $n$.

Now choose $\delta>0$ to be sufficiently small such that $\delta<\frac{1}{2(k-2)!(k-2)k\left(6k^2/c+6k\right)^{p-1}}$.
Then we get a contradiction, and this completes the proof.
\end{proof}
%%%%%%%%%%%%%%%%%%%%%%%%%%%%%%%%%%%%%%%%%

\section{Appendix B}
It is easy to see that $\nu(\mathcal{L}_i)=\nu+1$ holds for every $i\in[5]$.
So, it suffices to show that the families $\mathcal{L}_{i}$ are $d$-cluster-free.
\begin{claim}\label{claim-append-3-cluster-L1}
$\mathcal{L}_1$ is $3$-cluster-free and $\nu(\mathcal{L}_1)=\nu+1$.
\end{claim}
\begin{proof}
Suppose there exist three sets $L_1,L_2,L_3\in \mathcal{L}_1$ that form a $3$-cluster.
Since $L_1\cap L_2\cap L_3=\emptyset$, one of these three sets must be $C_i$ for some $i$, and
 we may assume that $L_1=C_1$.
On the other hand, since $|L_1\cup L_2\cup L_3|\le 2k$, $L_2 \text{ and } L_3$ must both contain $y$,
and $L_2\cap J, L_3\cap J$ must be both contained in $P_1$.
However, in this case, we would have $v_1\in L_1\cap L_2\cap L_3$, a contradiction.
Therefore, $\mathcal{L}_1$ is $3$-cluster-free.
\end{proof}

\begin{claim}\label{claim-append-3-cluster-L2}
$\mathcal{L}_2$ is $3$-cluster-free and $\nu(\mathcal{L}_2)=\nu+1$.
\end{claim}
\begin{proof}
Suppose there exist three sets $L_1,L_2,L_3\in \mathcal{L}_1$ that form a $3$-cluster.
Similarly to Claim \ref{claim-append-3-cluster-L1}, we may assume that $L_1=C_1$.
Since $|L_1\cap L_2\cap L_3|\le 2k$, we know that  $L_2\cap J$ and $L_3\cap J$ must be both nonempty.
For every $i\in \{2,3\}$, let $\mathcal{L}_2(i)=\{L\in \mathcal{L}_2: |L\cap J|=i\}$.
From the proof of Claim \ref{claim-append-3-cluster-L1}, we know that $L_2$ and $L_3$ cannot be both in $\mathcal{L}_2(2)$.

If $L_2\in \mathcal{L}_2(2)$ and $L_3\in \mathcal{L}_2(3)$,
then we would have $|L_2\cap L_3|\le k-3$ and this implies
$|L_1\cup L_2\cup L_3|= 3k-(|L_1\cap L_2|+|L_1\cap L_3|+|L_2\cap L_3|)\ge 2k+1$, a contradiction.

So we may assume that  $L_2,L_3$ are both contained in $\mathcal{L}_2(3)$.
Let $I_2=L_2\cap J$ and $I_3=L_3\cap J$.
By definition of $\mathcal{L}_{2}$, we have $|I_2\cap I_3|\le 1$.
Note that at least  one of $L_1\cap I_2, L_1\cap I_3, I_2\cap I_3$ must be the empty set,
since otherwise we would have $L_1\cap L_2\cap L_3\neq \emptyset$, a contradiction.
Therefore, we have $|L_1\cap L_2|+|L_1\cap L_3|+|L_2\cap L_3|\le k-3+2=k-1$,
 and this implies that $|L_1\cup L_2\cup L_3|\ge 2k+1$,
a contradiction.
Therefore, $\mathcal{L}_2$ is $3$-cluster-free.
\end{proof}

\begin{claim}\label{claim-append-4-cluster-L2}
$\mathcal{L}_3$ is $4$-cluster-free and $\nu(\mathcal{L}_3)=2$.
\end{claim}
\begin{proof}
Suppose there exist four sets $L_1,L_2,L_3,L_4\in \mathcal{L}$ that form a $4$-cluster.
Similarly, we may assume that $L_1=C_1$.
Since $|L_1\cup \cdots \cup L_4|\le 2k$, there are at least two sets in $\{L_2,L_3,L_4\}$ containing $v$.
We may assume that $v\in L_2$ and $v\in L_3$.
Let $E_2=L_2\cap W$ and $E_3=L_3\cap W$ and note that $E_2,E_3\in \mathcal{G}$.
Since $\mathcal{G}$ is $P_{2}^{k-2}$-free, $|E_2\cup E_3|\ge k$,
and this contradicts our assumption that $|L_1\cup \cdots \cup L_4|\le 2k$.
Therefore, $\mathcal{L}_3$ is $4$-cluster-free.
\end{proof}

\begin{claim}\label{claim-append-4-cluster-L4}
$\mathcal{L}_4$ is $4$-cluster-free and $\nu(\mathcal{L}_4)=\nu+1$.
\end{claim}
\begin{proof}
Suppose there exist four sets $L_1,L_2,L_3,L_4\in \mathcal{L}_4$ that form a $4$-cluster.
Similarly, we may assume that $L_1=C_1$.
Since $|L_1\cup L_2\cup L_3\cup L_4|\le 2k$, $L_2,L_3,L_4$ must all contain $y$ and all have nonempty intersection with $J$.
For every $i\in \{2,3,4\}$,  let $E_i=L_i\cap J$ and let $S_i=L_i\cap W$.
$|L_1\cup L_2\cup L_3\cup L_4|\le 2k$ implies that $|S_2\cup S_3\cup S_4|\le k-2$.

Suppose that $|S_2\cup S_3\cup S_4|= k-2$.
Again, $|L_1\cup L_2\cup L_3\cup L_4|\le 2k$ implies that $E_2=E_3=E_4$,
and $E_i\cap C_1\neq \emptyset$ holds for every $i\in \{2,3,4\}$.
However, in this case,  we would have $L_1\cap L_2\cap L_3\cap L_4\neq \emptyset$, a contradiction.

Therefore, we may assume that $S_2=S_3=S_4$.
Since $|L_1\cup L_2\cup L_3\cup L_4|\le 2k$, at least two sets in $\{E_2, E_3,E_4\}$
have nonempty intersection with $C_1$, and we may assume that $E_2\cap C_1\neq\emptyset$ and  $E_3\cap C_1\neq\emptyset$.
Now we already have $|L_1\cup L_2\cup L_3|=2k$, therefore, $E_4$
must be contained in $E_2\cup E_3\cup C_1$.
However, by the definition of $G$, this is impossible.
Therefore, $\mathcal{L}_4$ is $4$-cluster-free.
\end{proof}

\begin{claim}\label{claim-append-d-cluster-L5}
$\mathcal{L}_5$ is $d$-cluster-free and $\nu(\mathcal{L}_5)=\nu+1$.
\end{claim}
\begin{proof}
Suppose there exist $d$ sets $L_1,\ldots,L_d\in\mathcal{L}_5$ that form a $d$-cluster.
Similarly, we may assume that $L_1=C_1$.
For every $i\in \{2,\ldots,d\}$, let $S_i=L_i\cap W$ and $T_i=L_i\cap J$, and note that some of the $T_i$'s may be empty.
$|L_1\cup \cdots\cup L_d|\le 2k$ implies that $L_2,\ldots,L_{d}$ all contains $y$ and $|S_2\cup \cdots \cup S_{d}|\le k-1$.
Let $S=S_2\cup \cdots \cup S_{d}$.

If $S$ is of size $k-1$, then $T_i \subset C_1$ holds for every $i\in \{2,\ldots,d\}$,
here empty set is also considered as contained in $C_1$.
Since at most one set in $\{T_2,\ldots,T_d\}$ is empty, $S$ contains at least $d-2$ edges of $\mathcal{G}_1$ and, hence, $\mathcal{G}_1[S]\in\mathcal{H}_{k-1}^{d-2}$, a contradiction.

Therefore, we may assume that $S_2=\cdots=S_{d}$.
Note that $S$ is of size $k-2$ and every $T_i$ is nonempty.
Since $|L_1\cup \cdots\cup L_d|\le 2k$, at most one set in $\{T_2,\ldots,T_d\}$ is not contained in $C_1$.
This implies that at least $d-2$ sets in $\{T_2,\ldots,T_d\}$ are contained in $C_1$.
However, this implies that $S$ is an edge in $\mathcal{G}_1$
with multiplicity at least $d-2$, a contradiction.
Therefore, $\mathcal{L}_5$ is $d$-cluster-free.
\end{proof}

%%%%%%%%%%%%%%%%%%%%%%%%%%%%%%%%%%%%%%%%%%%%%%%%%%%%%%%%%%%%%%%%%%%%%%%%%%%%%%%%%%%%%%%%%%%
%\bibliographystyle{unsrt}
\bibliographystyle{abbrv}
\bibliography{clustermatch}
\end{document}